\documentclass{amsart}
\usepackage{graphicx}
\usepackage{graphicx,esint}
\usepackage{amssymb,amscd,amsthm,amsxtra}
\usepackage{latexsym}
\usepackage{epsfig}

\newtheorem{thm}{Theorem}[section]
\newtheorem{cor}[thm]{Corollary}
\newtheorem{lem}[thm]{Lemma}
\newtheorem{prop}[thm]{Proposition}
\theoremstyle{definition}
\newtheorem{defn}[thm]{Definition}
\theoremstyle{remark}

\numberwithin{equation}{section}

\newcommand{\R}{\mathbb R}

\newcommand{\eps}{\varepsilon}

\newcommand{\p}{\partial}

\newcommand{\comment}[1]{}

\newcommand{\st}{\,{\mbox{ s.t. }}\,}

\newcommand{\J}{\mathcal{J}}

\begin{document}

\title[A fully nonlinear problem with free boundary in the
plane]{A fully nonlinear problem with free boundary\\
in the plane}
\author{Daniela De Silva}
\address{Department of Mathematics, Barnard College, Columbia University, New York, NY 10027}
\email{\tt  desilva@math.columbia.edu}
\author{Enrico Valdinoci}%
\address{Dipartimento di Matematica, Universit\`a di Roma Tor Vergata, Roma, 00133, Italy
 }%
\email{\tt   enrico.valdinoci@uniroma2.it}%


\begin{abstract} We prove that bounded solutions to an overdetermined fully nonlinear free boundary problem
in the plane are one dimensional. Our proof relies on maximum
principle techniques and convexity arguments.
\end{abstract}
\maketitle

\section{Introduction}

Let $\phi_0$, $\phi_1\in C^2 (\R)$ be such that
$\phi_0(t)<\phi_1(t)$ for any $t\in \R$  and let $\Omega$ be the
open set in $\R^2$ trapped between the graphs of $\phi_0$ and
$\phi_1$, i.e.
$$ \Omega:=\big\{ (x_1,x_2)\in \R^2 \st
\phi_0(x_1)< x_2<\phi_1(x_1)\big\}.$$ Define \begin{align*} \J_0
&:=\big\{ (x_1,x_2)\in\R^2\st x_2 = \phi_0(x_1)\big\}, \\
\J_1&:= \big\{ (x_1,x_2)\in\R^2\st x_2=\phi_1(x_1)
\big\}.\end{align*} Notice that $\partial\Omega=\J_0\cup \J_1$.

We consider the following problem
\begin{equation}\label{1}
\left\{
\begin{matrix}
F(D^2 u)=0 & {\mbox{in $\Omega
$,}} \\
u=0 &
{\mbox{on $\J_0$,}}\\
u=1 & {\mbox{on $\J_1$,}}
\end{matrix}
\right.
\end{equation}
where $F$ is a uniformly elliptic fully nonlinear operator with
ellipticity constants $0<\lambda \leq \Lambda$, and $F(0)=0$ (see
\cite{CC} for the definition).

In this paper we prove the following symmetry result about bounded
solutions to the one-phase free boundary problem associated to
\eqref{1}.

\begin{thm}\label{main}
Let $u\in C^2(\Omega)\cap C^1 (\overline\Omega)$ be a solution to
\eqref{1}. Suppose that there exist $c_0,c_1 \in\R$ such that
\begin{equation}\label{L0}
{\mbox{$|\nabla u(x)|=c_0$ for any $x\in \J_0$,}}
\end{equation}
\begin{equation}\label{L1}
{\mbox{$|\nabla u(x)|=c_1$ for any $x\in \J_1$}}.
\end{equation}
Assume also that
\begin{equation}\label{mono}
0<u<1 \,{\mbox{ for any $x\in\Omega$.}}
\end{equation}
Then $c_0=c_1$ and $u$ is a linear function in $\Omega$. In
particular $\J_0$ and $\J_1$ are straight lines.
\end{thm}

Without loss of generality, using Hopf Lemma and a dilation, we
will assume that $c_1=1$.

Problem \eqref{1} may be seen as the fully nonlinear analogue of
the ideal
fluid jet model of \cite{AC}. In this sense,
the PDE in \eqref{1} may be seen as an incompressibility
condition and
the level
sets of
$u$ correspond to the stream lines along which the particles of
the fluid move, and our assumptions say that the stream lines
$\J_0$ and $\J_1$ which comprise the fluid jet
are nice curves.

Also, conditions \eqref{L0} and
\eqref{L1}
may be seen as pressure conditions on the boundary
of the fluid, due to Bernoulli's law.

Differently from the case of \cite{AC}, here the continuity
equation for the incompressible fluid is not given by the standard
Laplacian operator but, more generally, by a fully nonlinear
elliptic one.


Thus, Theorem \ref{main} may be seen as a rigidity result on an
fully nonlinear fluid jet model which determines the shape of the
fluid (i.e., the domain $\Omega$) and the stream lines of the
fluid (i.e., the function $u$), given some information on the
exterior pressure (namely, conditions \eqref{L0} and \eqref{L1}).

The inspiration for Theorem \ref{main} came from the work of the
first author about rigidity results for fully nonlinear phase
transition models \cite{DS} and from the work of the second author
on overdetermined semilinear problems \cite{FV}. We refer to these
papers for further motivation and related questions.
For different rigidity results for fully nonlinear operators
see also~\cite{DM}.

Of course, it would be interesting to know whether the analogue of
Theorem \ref{main} remains true in higher dimensions.

In order to prove our Theorem \ref{main} we show that a given
level set $\{u=\sigma\}$ is contained in a strip of arbitrarily
small width. The main ingredients in the proof are a Harnack type
inequality for the level set $\{u=\sigma\}$ (see Proposition
\ref{Harnack}) and a convexity result which roughly says that if
two balls at unit distance from the $\sigma$ level set of $u$ are
contained in $\{u < \sigma\}$, then their convex hull is also
contained in $\{u < \sigma\}$ (see Proposition \ref{convexhull}).
To obtain these two results, we use the maximum principle together
with appropriate radially symmetric barriers.

The rest of the paper is devoted to the proof of Theorem
\ref{main}. Namely, in Section~2 we introduce some notation and
tools which will be used throughout the paper. In Section~3, we
prove that a solution $u$ as in Theorem \ref{main} must satisfy
the free boundary condition $|\nabla u|=1$ on $\J_0.$ In
Section~4, we prove a Harnack type inequality for the level sets
of $u$ which will be used to show an improvement of flatness for
the $\sigma$ level set. In Section 5, we prove the desired
convexity property mentioned above and we exhibit the proof of
Theorem \ref{main}.

We conclude our introduction with a remark. In Theorem \ref{main}
we do not need to consider classical solutions to the free
boundary problem. Our result still holds for viscosity solutions
(see \cite{CS} for the definition) as long as $\J_0$ and $\J_1$
are continuous graphs.

\section{Main tools}

\subsection{Useful barriers.} In this section we construct suitable
super/subsolutions, by modifying the one-dimensional solutions to
\eqref{1}. The barriers we construct here are inspired by the ones
introduced in \cite{S} and developed in \cite{DS, V1, V2}.
Throughout the paper, constants depending only on the ellipticity
constants $\lambda, \Lambda$ will be called universal constants.




\begin{prop}\label{bar}
Let $\sigma \in[0,1)$. There exist universal constants $C_i>0$
such that if $R \geq C_0$ and
\begin{eqnarray*}&g:[0,\rho_1]\to \R, \quad g(\rho_1)=1, \quad \rho_1 \leq 2 \\& g(s):= \sigma + \left(1+\dfrac{C_1}{R}\right)s
-\dfrac{C_2}{R}s^2\end{eqnarray*} then
\begin{equation*}\beta^+_{R,\sigma}(x):= g(|x|-R)\in C^2(\overline{B_{R_1} \setminus B_{R}}), \quad R_1:=R+\rho_1\end{equation*}
satisfies
\begin{itemize}
\item[{$(i)$\ }] $\beta^+_{R,\sigma}$ is radially
strictly increasing with $$\beta^+_{R,\sigma} =\sigma \quad
\mbox{on $\partial B_{R},$} \quad \beta^+_{R,\sigma} =1 \quad
\mbox{on $\partial B_{R_1}$};$$
\item[{$(ii)$\ }] $F(D^2 \beta^+_{R,\sigma})<0$
in $B_{R_1}\setminus B_{R}$;
\item[{$(iii)$\ }] $|\nabla \beta^+_{R,\sigma}
|\le1+C_3/R$ on $\partial B_{R}$ and $|\nabla \beta^+_{R,\sigma}|
\ge1+C_4/R$ on ${\partial B_{ R_1}}$.

\end{itemize}
\end{prop}

\begin{proof} First let us compute $\rho_1$, that is ($R$ large)
 $$\rho_1 = \frac{R+C_1}{2C_2}\left( 1-\sqrt{1-\frac{4 C_2
R(1-\sigma)}{(R+C_1)^2} } \right).$$ Since for small $t \geq 0$,
\begin{equation*}\label{sqrt}1 - \sqrt{1-t}\le t\end{equation*} we easily obtain that for $R$ large
$\rho_1 \leq 2.$ Thus property $(i)$ follows immediately by
choosing $C_1 > 4C_2$.

To check that $(ii)$ holds, we let $s:=|x|-R$ and we notice that
for an orthogonal matrix $O$ we have
\begin{eqnarray*}
F(D^2 \beta^+_{R,\sigma}(x))&=& F\left(O^t\left(
\begin{matrix}
g''(s) & 0 \\ 0 & g'(s)/|x|
\end{matrix}\right)O
\right)\\
&=& F\left(O^t\left(
\begin{matrix}
-2C_2/R & 0 \\ 0 & \big[(1 +C_1/R)-2C_2 s/R \big]/|x|
\end{matrix}\right)O
\right)\\
&\le& \frac{\Lambda}{|x|} \left|1 +\frac{C_1}{R}-2\frac{C_2}{R} s
\right|-\frac{2\lambda C_2}{R}
\\
&\le& \frac{2\Lambda}{R}-\frac{2\lambda C_2}{R}<0
\end{eqnarray*} for $|x| > R$, large $R$ and $C_2 > \Lambda/\lambda.$

Property $(iii)$ also follows immediately, since $$g'(0)=1+C_1/R
\quad  \text{and} \quad g'(\rho_1)= 1+ C_1/R - 2C_2\rho_1/R \geq 1
+ C_2/R$$ as long as $C_1 \geq 5 C_2.$
\end{proof}

Similar arguments also give the following Proposition.

\begin{prop}\label{bar2}
Let $\sigma \in[0,1)$. There exist universal constants
$\tilde{C}_i>0$  such that if $R \geq \tilde C_0$ and
\begin{eqnarray*}&
\tilde g:[\tilde \rho_1,0]\to \R, \quad \tilde g(\tilde \rho_1)=1,
\quad \tilde \rho_1 \geq -3\\& \tilde g(s):= \sigma -
\left(1-\dfrac{\tilde C_1}{R}\right)s -\dfrac{\tilde
C_2}{R}s^2\end{eqnarray*} then
\begin{equation*}\beta^-_{R,\sigma}(x):= \tilde g(|x|- R)\in C^2(\overline{B_{R} \setminus B_{\tilde R_1}}), \quad \tilde R_1:=R+\tilde \rho_1\end{equation*}
satisfies
\begin{itemize}
\item[{$(i)$\ }] $\beta^-_{R,\sigma}$ is radially
strictly decreasing with $$\beta^-_{R,\sigma} =\sigma \quad
\mbox{on $\partial B_{R},$} \quad \beta^-_{R,\sigma} =1 \quad
\mbox{on $\partial B_{\tilde R_1}$};$$
\item[{$(ii)$\ }] $F(D^2 \beta^-_{R,\sigma})>0$
in $B_{R}\setminus B_{\tilde R_1}$;
\item[{$(iii)$\ }] $|\nabla \beta^-_{R,\sigma}
|\ge1-\tilde C_3/R$ on $\partial B_{R}$ and $|\nabla
\beta^-_{R,\sigma}| \le 1-\tilde C_4/R$ on ${\partial B_{ \tilde
R_1}}$.\end{itemize}
\end{prop}




When $\sigma\in (0,1)$, it is useful to extend the barrier
$\beta^+_{R,\sigma}$ also to values below $\sigma$, as in the
following result.

\begin{prop}\label{bar tilde}
Let $\sigma \in (0,1)$. There exist universal constants $\bar
C_i>0$ such that if $R \geq \bar C_0$ and
\begin{eqnarray*}&\bar g:[\rho_0,\rho_1]\subset [-2,2]\to \R, \quad \bar g(\rho_0)=0, \quad \bar g(\rho_1)=1, \quad
\rho_0 < 0 < \rho_1\\
& \bar g(s):= \sigma + \left(1+sign(s)\dfrac{ \bar C_1}{R}\right)s
-\dfrac{\bar C_2}{R}s^2\end{eqnarray*} then
\begin{eqnarray*}&\bar \beta_{R,\sigma}(x):= \bar g(|x|-R)\in C^2((\overline{B_{R_1} \setminus B_{R_0}}) \setminus \partial
B_R),\\
& R_1:=R+\rho_1, \quad R_0: = R + \rho_0, \end{eqnarray*}
satisfies
\begin{itemize}
\item[{$(i)$\ }] $\bar \beta_{R,\sigma}$ is radially
strictly increasing with $$\bar \beta_{R,\sigma} =0 \quad \mbox{on
$\partial B_{R_0}$}, \quad  \bar \beta_{R,\sigma} =\sigma \quad
\mbox{on $\partial B_{R},$} \quad \bar \beta_{R,\sigma} =1 \quad
\mbox{on $\partial B_{R_1}$};$$
\item[{$(ii)$\ }] $F(D^2 \bar \beta_{R,\sigma})<0$
in $(B_{R_1}\setminus B_{R_0}) \setminus \p B_R$;
\item[{$(iii)$\ }] $|\nabla \bar \beta_{R,\sigma}
|\le1 - \bar C_3/R$ on $\partial B_{R_0}$ and $|\nabla \bar
\beta_{R,\sigma}| \ge1+ \bar C_4/R$ on ${\partial B_{ R_1}}$.

\end{itemize}
\end{prop}

\begin{proof}
In view of Proposition \ref{bar}, we only need to focus on the
case $s \in [\rho_0, 0]$ and check that the properties above hold
by possibly choosing the constants $\bar C_1, \bar C_2$ larger
than $C_1, C_2$. The proof follows from similar computations as in
the proof of Proposition \ref{bar}. We sketch it for completeness.

Property $(i)$ is obvious, as long as $R$ is large.

Again, one can easily compute that
$$\rho_0 = \frac{R+\bar C_1}{2C_2}\left(1- \sqrt{1 + 4\frac{\bar C_2 \sigma}{(R+\bar C_1)^2} }\right).$$
Since for $t \geq 0$ we have
$$1-\sqrt{1+t} \geq -t,$$ we can estimate that $$\rho_0 \geq
-2.$$

To check that $(ii)$ holds, again we let $s:=|x|-R$ then, for an
orthogonal matrix $O$, we have
\begin{eqnarray*}
F(D^2 \bar \beta_{R,\sigma}(x))&=& F\left(O^t\left(
\begin{matrix}
\bar g''(s) & 0 \\ 0 & \bar g'(s)/|x|
\end{matrix}
\right)O\right)\\
&\le& \frac{\Lambda}{|x|} \left|1 -\frac{\bar C_1}{R}-2\frac{\bar
C_2}{R} s \right|-\frac{2\lambda \bar C_2}{R}
\\
&\le& \frac{2\Lambda}{R}-\frac{2\lambda \bar C_2}{R}<0
\end{eqnarray*} for $|x| > R_0$, large $R$, and $\bar C_1 > 4 \bar C_2, \bar C_2 > \Lambda/\lambda.$

Property $(iii)$ also follows immediately, since $$\bar g'(\rho_0
)=1-\bar C_1/R - 2 \bar C_2   \rho_0/ R \le 1 - \bar C_2/R$$ as
long as $\bar C_1 \geq 5 \bar C_2.$
\end{proof}

{F}rom now on we extend $\bar \beta_{R,\sigma}$ to be 0 in
$B_{R_0}.$ Also, sometimes we think that $\bar \beta_{R,\sigma}$
is extended to 1 outside of $B_{R_1}$. This will be clear from the
context.

\

\noindent \textbf{Remark.} Notice that, since $-2 \leq \rho_0 <
\rho_1 \leq 2$ one has
\begin{equation}\label{gbound}|\bar g(s) - (s+\sigma)| \leq
\frac{\bar C_5}{R}, \quad s \in [\rho_0, \rho_1].\end{equation}In
particular,
evaluating \eqref{gbound} at $s=\rho_0$
and at $s=\rho_1$,
we see that
\begin{equation}\label{2.1a}
\begin{split}&
R -\sigma- \frac{\bar
C_5}{R} \leq R_0 \leq R -\sigma +\frac{\bar C_5}{R},\\
&
R+ (1-\sigma)- \frac{\bar C_5}{R}
\leq R_1 \leq R+ (1-\sigma)+\frac{\bar C_5}{R}.\end{split}
\end{equation}
Define $T_0^1$ to be the truncation at levels 0 and 1, i.e.
\begin{equation}\label{t01}T_0^1(t):= \begin{cases}0 & \text{if $t
< 0 $}\\
t & \text{if $0 \leq t \leq 1$}\\ 1 & \text{if $t>1.$}
\end{cases}\end{equation} Then from \eqref{gbound} it follows that $$T_0^1(s + \sigma + C_5/R) \geq \bar
g(s) \quad \text{on $[\rho_0, \rho_1],$}$$ and hence
\begin{equation}\label{bbound} T_0^1(|x|-R + \sigma + C_5/R) \geq
\bar \beta_{R,\sigma}(x) \quad \text{in $\R^2$}.
\end{equation}

\

 Clearly, using similar arguments, also the
barrier $\beta^-_{R,\sigma}$ can be extended below $\sigma$ to a
barrier $\underline {\beta}_{R,\sigma}$. We omit its precise
definition, since we do not explicitly need it here.

\subsection{Sliding method.} In what follows, the barriers constructed above and
one-dimensional solutions to \eqref{1} will be used as comparison
functions in what we call the sliding method. For the sake of
clarity, we fix here some notation and terminology which we will
use throughout the rest of the paper.

Let $u$ be a solution to \eqref{1}-\eqref{mono}. {F}rom now on any
such $u$ is extended to be 0 below $\J_0$ and 1 above $\J_1$ (when
not specified ``below" and ``above" are intended with respect to
the $e_2$ direction).

Given a continuous function $v : K \to \R$, $K$ compact, we set
\begin{equation}\label{translate}
v_t(x):= v(x-te_2), \quad t \in \R.
\end{equation}
We say that $v$ is above $u$ if $v \geq u$ in $K$. Also, we say
that $v$ touches $u$ from above if $v$ is above $u$ and $v_t$ is
not above $u$ for any $t>0.$ Clearly, this implies that
$u(x_0)=v(x_0)$ at some $x_0 \in K \cap \overline{\Omega}$ which
we call a contact point.

The sliding method can be roughly described as follows. Assume
that for some $t_0$ we have that $v_{t_0} \geq u.$ Then by
continuity, there exists the smallest $\bar t \geq t_0$ such that
$v_{\bar t}$ touches $u$ from above. If $v \neq u$ is a
supersolution to our equation, then we can apply the maximum
principle to conclude that contact points are in $\p K$. If we
also know for example that $|\nabla v| \leq c_0$ (resp. $|\nabla
v| \geq 1 $) on $\{v =0\}$ (resp. on $\{v=1\}$) then we can
conclude from Hopf Lemma that no contact points occur on $\{v=0\}$
(resp. on $\{v=1\}$) unless $v=u$.

This method will be a key tool in all our proofs.

\section{Determining $c_0$}

The purpose of this section is to determine $c_0$, that is the
exterior pressure of the fluid on $\J_0$, by knowing the pressure
on $\J_1$. Precisely, we prove the following Proposition.

\begin{prop}\label{det c}
Let $u\in C^2(\Omega)\cap C^1 (\overline\Omega)$ be a solution to
\eqref{1}-\eqref{mono}. Then $c_0=1$.
\end{prop}

\begin{proof}
We apply the sliding method with the comparison function
$v=\beta^+_{R,0}$ in Proposition \ref{bar} (we use the notation in
that Proposition).

Since $u=0$ below $\J_0$ and $v \geq 0$ in $B_{R_1}$ there exists
a $t_0 < 0$ such that $v_{t_0}$ is above $u$. Let $\bar t$ be the
smallest $t \geq t_0$ such that $v_{ t}$ touches $u$ from above.

According to Proposition \ref{bar}$(ii)$, $v_{\overline t}$ is
a strict supersolution in $B_{R_1} \setminus B_{R}$, thus by the
comparison principle if $\bar x$ is a contact point then
\begin{equation}\label{Si}
\overline x \in \partial\Big( B_{R_1}(0, \overline t)\setminus
B_{R}(0,\overline t)\Big) .\end{equation} We now show that
\begin{equation}\label{No}
\overline x \not\in \partial B_{R_1}(0, \overline t).
\end{equation}
Indeed suppose by contradiction that $ \overline x \in \partial
B_{R_1}(0, \overline t). $ Then, by Proposition \ref{bar}$(i)$, we
have $v_{\overline t}(\overline x)=1$ and so since $0<u<1$ in
$\Omega$,
\begin{equation}\label{8s8} \overline x \in \J_1.\end{equation} Also, if
$\nu$ is the exterior normal of $\partial B_{R_1}(0,\overline t)$
we get
\begin{equation*}\label{8s9}
\partial_\nu v_{\overline t}(\overline x)
\le \partial_\nu u(\overline x).\end{equation*} Then, the
inequality above together with Proposition \ref{bar}$(i)$,
\eqref{L1}, \eqref{8s8}, give
$$ |\nabla v_{\overline t}(\overline x)|=
\partial_\nu v_{\overline t}(\overline x)
\le 1.$$  Since $ \overline x \in \partial B_{R_1}(0, \overline
t)$, this inequality
contradicts Proposition \ref{bar}$(iii)$, and
so \eqref{No} is proved.

Therefore by \eqref{Si} and \eqref{No} we obtain that
\begin{equation}\label{09}\overline x
\in \partial B_{ R}(0, \overline t) .\end{equation} Thus,
by Proposition \ref{bar}$(i)$, we have that $\overline x\in \J_0$
and
so, by \eqref{L0}, \eqref{09} and Proposition \ref{bar}$(i, iii)$,
we conclude that
\begin{equation}\label{gg}
c_0=|\nabla u(\overline x)|\le |\nabla v_{\overline t}(\overline
x)|\le 1+\frac{C_3}{R}.
\end{equation}

Now we perform a (upside-down) sliding argument with comparison
function $w=\beta^-_{R,0}$ (recall Proposition~\ref{bar2}).

Since $u=1$ above $\J_1$, there exists $t_0>0$ such that $u$ is
above $w_{t_0}.$ Let $\underline t$ be the largest $t \leq t_0$
such that $u$ touches $w_{t}$ from above.

According to Proposition \ref{bar2}$(ii)$  $w_{\underline t}$ is
a strict subsolution thus, by the comparison principle, we know
that if $\underline x$ is a contact point then
\begin{equation}\label{Si2}
\underline x \in \partial\Big( B_R(0,\underline t) \setminus B_{
\tilde R_1}(0,\underline t)\Big) .\end{equation} As before, we
deduce that
\begin{equation}\label{No2}
\underline x \not\in \partial B_{\tilde R_1}(0,\underline t)
.\end{equation} Indeed, if \eqref{No2} were false, we would have
that $u(\underline x)=1$ thanks to Proposition \ref{bar2}$(i)$.
Accordingly, from \eqref{L1} and Proposition \ref{bar2}$(iii)$,
$$ 1=|\nabla u(\underline x)|
\le |\nabla w_{\underline t}(\underline x)|\le 1-\frac{\tilde
C_3}{R}.$$ This contradiction proves \eqref{No2}.

Then, from \eqref{Si2} and \eqref{No2}, we deduce that $\underline
x \in \partial B_{ R}(0,\underline t)$ and so, from Proposition
\ref{bar2}$(i)$ and \eqref{mono}, that $\underline x\in \J_0$.

Finally, \eqref{L0} and Proposition \ref{bar2}$(i, iii)$ give that
$$ c_0=|\nabla u(\underline x)|
\ge |\nabla w_{\underline t}(\underline x)|\ge 1-\frac{\tilde
C_4}{R}.$$ This and \eqref{gg} imply that $c_0=1$, by taking $R$
arbitrarily large.
\end{proof}

\section{Level set analysis}

In this section we wish to prove a Harnack type inequality for the
level sets of a solution $u$ to our problem. The techniques we use
here and in the next section have been inspired by \cite{DS}.

{F}rom now on we denote by $u$ a solution to
\begin{equation}\label{2}
\left\{
\begin{matrix}
F(D^2 u)=0 & {\mbox{in $\Omega
$,}} \\
u=0, |\nabla u|=1 &
{\mbox{on $\J_0$,}}\\
u=1, |\nabla u|=1 & {\mbox{on $\J_1$.}}
\end{matrix}
\right.
\end{equation}
with $0<u<1$ in $\Omega$. As usually $u$ is extended to be 0 below
$\J_0$ and 1 above $\J_1.$

As in the case of the one-phase problem in \cite{CS}, one obtains
that $u$ is Lipschitz continuous with universal bound.

\begin{lem}\label{Lip}$u$ is uniformly Lipschitz continuous, i.e $$\|\nabla u\|_{L^\infty} \leq K$$ with $K$ universal constant.\end{lem}

\begin{proof}
Let $x_0 \in \Omega,$ and $d=\min\{\text{dist}(x_0,\J_0),
\text{dist}(x_0,\J_1)\}$. Assume without loss of generality that
$d=\text{dist}(x_0,\J_0).$ We wish to prove that
\begin{equation*}\label{usmall}|\nabla u(x_0)| \leq K.\end{equation*}

Let
$$v(x)=\frac{1}{d}u(x_0 +dx)$$
be the rescale of $u$ in $B_d(x_0)$. Then $v \geq 0$ solves a
uniformly elliptic equation
\begin{equation}\label{G}G(D^2v)=0 \ \ \text{in} \ \ B_1(0),\end{equation} with $G(M)= d F(M/d)$ having the same ellipticity
constants as $F,$ and $G(0)=0.$ Hence, by Harnack's inequality
(see \cite{CC})  \begin{equation}
\label{+0} v \geq c v(0) \ \ \text{in} \ \ 
B_{1/2}(0).\end{equation}

Let us choose $\beta<0$ such that, the radially symmetric function
$$g(x)=\frac{c v(0)}{2^{-\beta} -1}(|x|^\beta -1)$$satisfies $G(D^2g) \geq 0$ in the annulus $B_1 \setminus
\overline{B_{1/2}}$, $g = 0$ on $\partial B_1$ and $g = c v(0)\le 
v$ on
$\partial B_{1/2},$ due to~\eqref{+0}.

Then, by the maximum 
principle
$$v \geq g \ \
\text{in} \ \ B_1 \setminus \overline{B_{1/2}}.$$ Now, let $x_1
\in
\partial B_1(0)$ be such that $v(x_1)=0$. Then, since $\nabla
v(x)= \nabla u(x_0+dx)$, and $u$ solves \eqref{2}, we have
$|\nabla v(x_1)| = 1.$ Let $\nu$ be the inward normal to $\partial
B_1$ at $x_1$. 

Then, 
$$1 = |\nabla v(x_1)| \geq v_{\nu}(x_1) \geq g_{\nu} (x_1)
\geq C 
v(0).$$ Using Harnack's inequality we conclude that
\begin{equation}\label{+A}
\|v\|_{L^\infty(B_{1/2})} \leq C.
\end{equation} Also,
by the $C^{1,\alpha}$ estimate for equation \eqref{G} (see
\cite{CC}), we have that \begin{equation}\label{+B}
\|v\|_{C^{1,\alpha}(B_{1/4})} \leq C \|v\|_{L^\infty(B_{1/2})}.
\end{equation} Therefore, from~\eqref{+A} and~\eqref{+B},
we obtain that $|\nabla v(0)|=|\nabla u (x_0)|$ is bounded by a
universal constant.
\end{proof}

Now we fix a $\sigma \in (0,1)$ and we proceed to analyze the
properties of the correspondent level set of $u$, that is~$\{
u=\sigma\}$.

We start
with a
Definition and an elementary Lemma and then we state and prove the
desired Harnack type inequality. In the next section, towards
proving our Theorem \ref{main}, we will show that $\{u=\sigma\}$
is flat enough, i.e. it is contained in a strip. Then, Harnack
inequality
will be used to prove an improvement of flatness for
$\{u=\sigma\}$ which implies the result of Theorem \ref{main}.


\begin{defn}We define $A_{R,\sigma}$ to be the set of all $x$'s such that $$B_{R}(x) \subseteq \{ u<\sigma\}\quad \text{and
\quad $\overline{B_{R}}(x+te_2) \subset \{u< \sigma\}$} \quad
\text{for all $t < 0$}.$$
\end{defn}

In other words, $A_{R,\sigma}$ consists of the centers of the
balls of radius $R$ whose translates in the direction $-e_2$ are
compactly included in $\{u<\sigma\}.$

\begin{lem}\label{easyHarnack}
Assume that $x_* \in A_{R,\sigma}$, $R$ large, then
\begin{itemize}\item[$(i)$] $ u(x)\le \bar\beta_{R,\sigma}(x-x_*);$
\item[$(ii)$] If $x_0 \in \p B_R(x_*) \cap \{u=\sigma\}$, then
$\J_1\cap B_{K}(x_0) \neq \emptyset$ for some universal constant
$K.$\end{itemize}
\end{lem}

\begin{proof}Claim $(i)$ follows easily by applying the sliding method with the
comparison function $\bar\beta_{R,\sigma}(x-x_*)$,
where the notation of
Proposition \ref{bar tilde} is used.

The argument
is similar to the proof of Proposition \ref{det c}. Touching along
$\J_0,\J_1$ is excluded by Proposition \ref{bar tilde}$(iii)$ and
contact points always occur on $\{u=\sigma\}$.

In order to prove $(ii)$, we observe that in Proposition \ref{bar
tilde} we can take $\sigma=1$ and $\rho_1=0$. Then, for $R$ large
we have that $\bar{\beta}_{R,1}$ is a strict supersolution in
$B_{R}\setminus B_{R_0}$ and it satisfies the strict free boundary
condition $|\nabla \bar \beta_{R,1} |<1$ on $\partial B_{R_0}.$

Now we fix $K$ large and apply the sliding method in the direction
$\nu$ of $x_0 - x_* $ with comparison function $v(x)=\bar
\beta_{K/2,1}(x-x_*)$. This means that the translates $v_t(x)=v(x
-t\nu)$ are in the $\nu$ direction.

According to $(i)$, $u=0$ in $B_{K/2}(x_*)$ ($R > K$), thus $v$ is
above $u$ in such ball. Let $\bar t$ be the smallest $t \geq 0$
such that $v_t$ touches $u$ from above. Since $v_{\bar t}$ is a
strict supersolution and it satisfies the strict free boundary
condition $|\nabla v_{\bar t}| < 1$ on its zero level set we
conclude that touching points can only occur on
\begin{equation}\label{empty}\p B_{K/2}(x_*+\bar t \nu) \cap \J_1\neq
\emptyset.\end{equation} In particular this implies,
$$\bar t \geq R- K/2.
$$ On the other hand, $\bar t \leq R$ since $0=v_R(x_0) < u(x_0)=\sigma.$
The bounds on $\bar t$ imply that $x_0 \in B_{K/2}(x_*+\bar t
\nu)$ which together with and \eqref{empty} yield the desired
result.
\end{proof}

\begin{lem}\label{Harnack}(Harnack Inequality) Assume that $$\partial_{x_2} u \geq 0 \quad
\text{in $\Omega.$}$$ If $x_*= -R\nu \in A_{R,\sigma}$ and
$$0 \in \{u=\sigma\} \cap \p B_R(x_*), \quad
\nu_2>0,$$ then, for any $M>0$ there exist $\bar R, \bar C$
depending on $\sigma, \nu_2,M, \lambda, \Lambda$ such that if $R
\geq \bar R$ then\begin{equation}\label{0.A}
B_{M-\frac{\bar C}{R}}(M\nu) \subset \{u > \sigma\}.
\end{equation}
\end{lem}

\begin{proof}
Let\begin{equation}\label{0.4}
w(x)= T_0^1(x\cdot \nu + \sigma +
\frac{C_1}{R}),\end{equation}
with $C_1$ to be
chosen later (see \eqref{t01} for the definition of $T_0^1$).

Notice that~$F(D^2 w)=0$
and~$|\nabla w|=1$ in~$\{0<w<1\}$.

First, we wish to show that if $R$ is large enough
\begin{equation}\label{wbig}w(x) \geq \bar \beta_{R,\sigma}(x-x_*)
\quad  \text{for any $x\in
B_{\frac{20M}{\nu_2}}(0)$}.\end{equation} Since
$T_0^1$ is non-decreasing, in view of \eqref{bbound} it suffices
to show that
$$x \cdot \nu +\sigma + C_1/R \geq |x-x_*|-R + \sigma + \bar C_5/R
\quad \text{in $B_{\frac{20M}{\nu_2}}(0)$.}$$ Thus we need to
prove that
\begin{equation}\label{4.444}
x \cdot \nu +  R+  C_2/R \geq |x-x_*|
\quad \text{for any $x\in
B_{\frac{20M}{\nu_2}}(0)$},\end{equation}
for some $C_2$ large.

For this we write
$$x=:a\nu + b \nu^\perp \in
B_{\frac{20M}{\nu_2}}(0), \quad |a|,|b| \leq 20M/\nu_2.$$
Since $x_*=-R\nu$,
\eqref{4.444} may be written as
$$a+ R + C_2/R \geq |(a+R)\nu + b \nu^\perp|$$ or, equivalently,
$$(a+R+ C_2/R)^2 \geq (a+R)^2+b^2.$$ Hence it is enough to prove
that $$(1+2a/R)C_2 \geq b^2.$$ Since $|a|,|b| \leq 20M/\nu_2$ this
inequality holds for $R$ and $C_2$ large depending on $M/\nu_2$.
Thus \eqref{wbig} is proved.

As a consequence, in view of Lemma \ref{easyHarnack}$(i)$,
we obtain that \begin{equation}\label{wu} w \geq u \quad  \text{in
$ B_{\frac{20M}{\nu_2}}(0)$}.
\end{equation}

Now, according to Lemma \ref{Lip} $u$ is uniformly Lipschitz with
universal Lipschitz constant $K$. Thus, if $$u(x) \in
I:=[\sigma/2, 1/2+\sigma/2]$$ then there exists a constant $c$
depending on $K,\sigma$ such that $$B_c(x) \subset \{0<u<1\}.$$

Clearly, the same statement holds also for $w$, that is $$B_c(x)
\subset \{0<w<1\}$$ and so~$F(D^2 w)=0$ in $B_c(x)$, as long
as~$w(x)\in I$.

In particular, we observe that~$u(0), w(0) \in I$, so $$
F(D^2w)=F(D^2u)=0 \quad \text{in $B_{c}(0).$}$$
Also, by \eqref{0.4},
$$w(0) - u(0)=C_1/R,$$
So, in view of \eqref{wu} we can apply the Harnack inequality in
$B_c(0)$ and obtain
$$w - u \leq C'/R \quad \text{in $B_{c/2}(0)$}.$$
Now, we
observe that~$\{w=\sigma\}$ is the line~$\{x\cdot\nu
=-C_1/R\}$, due to~\eqref{0.4}, and therefore we
pick $x_1 \in B_{c/2}(0) \cap \{w=\sigma\}$
(of course, this is possible since~$R$ is large enough).

Then,
\begin{equation}\label{x_1} 0 \leq w(x_1) -u(x_1) \leq
C'/R\end{equation} and  $u(x_1),w(x_1) \in I$. Arguing as above,
we conclude that
$$w - u \leq  C''/R \quad \text{in $B_{c/2}(x_1)$}.$$
Now we can pick $x_2 \in B_{c/2}(x_1) \cap \{w=\sigma\}$ and
iterate this argument a finite number of times to conclude that in
a $c/4$ neighborhood $N$ of $\{w=\sigma\} \cap B_{10M/\nu_2}(0)$
we have that \begin{equation}\label{Hwu} 0 \leq w - u \leq
C/R,\end{equation} with $C$ depending also on $\sigma, M/\nu_2.$

If $$x \in L:=\{x \cdot \nu = C/R\}\cap B_{10M/\nu_2}(0)$$ then
\begin{equation}w(x)= \sigma + \frac{C+C_1}{R}
.\end{equation}

Also, $x \in N$ if~$R$ is large,
and so from \eqref{Hwu} we get $$u(x)
> \sigma, \quad{\mbox{for any $
x\in L$.}}$$
Therefore,
since $u$ is monotone in the vertical direction, \eqref{0.A}
will follow if we show that
\begin{equation}\label{0.B}{\mbox{
$B_{M-\frac{\bar C}{R}}(M\nu)$ is above $L$ in the vertical
direction.}}\end{equation} For this, we choose $\bar C:= C$ and we
observe that~\eqref{0.B} is proved if we show that
\begin{equation}\label{x}
\mbox{$\forall \ x$ s.t. $|x-M\nu| \le M- C/R, \quad \exists \
\tilde{x} \in L$ s.t. $x-\tilde{x} = t e_2, \ t \geq
0.$}\end{equation}
Hence, in order to
prove~\eqref{x} (and so to complete the proof of the Lemma), given
$x$ as in~\eqref{x}, we let $$\tilde x := \left(x_1,
\frac{1}{\nu_2}\left(\frac{C}{R}-x_1 \nu_1\right)\right).$$
Clearly $\tilde x \cdot \nu = C/R$.

Moreover, since $|x|
< 2M$, then $|\tilde x_1|< 2M/\nu_2,$ and $|\tilde x_2| <
3M/\nu_2$ which implies that $\tilde x \in B_{10M/\nu_2}(0)$,
hence \begin{equation}\label{0900}\tilde x \in L.\end{equation}

Furthermore,
$$\left(x- \frac{C}{R}\nu\right) \cdot \nu \geq 0$$ which
immediately implies that $x_2 \geq \tilde{x}_2$. This
and~\eqref{0900} prove~\eqref{x}.
\end{proof}

\section{Convexity arguments and proof of Theorem \ref{main}}

The purpose of this section is to exhibit the proof of Theorem
\ref{main}. As mentioned in the previous section, our aim is
first to
show that $\{u=\sigma\}$ is flat
enough, that is, it lies inside a strip.

We start with two
Propositions
which will be the key tools to achieve our goal. Roughly, the
first Proposition gives the convexity of the set $\{u<\sigma\}$
while the second one gives the convexity of $\{u>\sigma\}$.
However, their statements are different since the first one
requires the existence of two balls contained in $\{u< \sigma\}$
while the second one requires the existence of two points on or
above $\J_1$.

\begin{prop}\label{convexhull}Assume $y_i \in A_{R,\sigma}, i
=1,2.$ Then the
convex envelope generated by the balls $B_{R- C/R}(y_i), i=1,2$ is
included in $\{u < \sigma\},$ for $R$ large and $C$ universal.
\end{prop}

\begin{proof}
Let $l$ be the common tangent line from above
to $B_{R - C/R}(y_i),
i=1,2$
and denote by $x_i, i=1,2$ the respective points of tangency.



Let $v$ be the one-dimensional linear solution to
$F(D^2v)=0$ in~$\{0<v<1\}$
with~$|\nabla v|=1$
which equals
$\sigma$ on the line $l$, and denote by $l_0$ and $l_1$ the 0 and
the 1 level set of $v$ respectively.

Call $z_i, \xi_i$ the
intersection points of $l_i$ and the lines passing through
$y_1,x_1$ and $y_2,x_2$ respectively ($i=0,1$). Also, let $D$ be
the open rectangle with vertices at $z_i,\xi_i$, $i=0,1$.

The segment joining any two points $\eta$ and $\zeta$
which contains $\zeta$ but not $\eta$ will be denoted by~$(\eta,
\zeta]$. If it contains both~$\eta$ and~$\zeta$, we write~$[
\eta,\zeta]$.

In this notation, we claim that
\begin{equation}\label{SC}{\mbox{
$v > u$ on $S:=(z_0,z_1] \cup (\xi_0,\xi_1]$, and
$z_0$ and $\xi_0$ are strictly below $\J_0$.
}}\end{equation}

We assume for the moment that the claim in~\eqref{SC}
holds, and we apply
the sliding method.

Since $u=0$
below $\J_0$ there exists a $t_0 < 0$ such that $v_{t}$ is above
$u$ for all $t \leq t_0$. Let $\bar t$ be the smallest $t \geq
t_0$ such that $v_t$ touches $u$ from above. Assume by
contradiction that $\bar t \leq 0$.

Then, since $y_i+ te_2 \in A_{R,\sigma}$ for all $t <0$, clearly
the
claim in~\eqref{SC} holds also for $v_{\bar t}$, that is
\begin{eqnarray}\label{sides}& v_{\bar t} > u \quad \mbox{on $S+\bar t
e_2$ and}\\
\nonumber & z_{\bar t}:=z_0+\bar t e_2, \xi_{\bar t}:=\xi_0+\bar t
e_2 \quad \mbox{are strictly below $\J_0$}.
\end{eqnarray}
Thus, no contact points can occur in $D_{\bar t}:= D+\bar t e_2$,
otherwise $u$ and $v_{\bar t}$ coincide and \eqref{sides} is
contradicted.

Also, in view of Proposition~\ref{det c}
and Hopf Lemma, contact points
cannot
occur on the 0 and 1 level set of $v_{\bar t}$ except at the
vertices of $D_{\bar t}$. Hence, using again \eqref{sides} we
conclude that
$$v_{\bar t}> u \quad \mbox{in $\overline{D}_{\bar t} \setminus
[z_{\bar t}, \xi_{\bar t}]$ and $[z_{\bar t}, \xi_{\bar t}]
\Subset \{u=0\}.$}$$ This implies that for a small $\eps
>0$ the translate $v_{\bar t + \eps}$ is above $u$ which
contradicts the definition of $\bar t$. Thus, $\bar t
>0$ and by the arguments above
$$v_{t}> u \quad \mbox{in $\overline{D}_{t} \setminus
[z_{t}, \xi_{t}]$ for all $t \leq 0.$}$$ In particular, since we
have chosen $v$ so that its $\sigma$ level set coincides with $l$,
we obtain that $u<\sigma$ below $[x_1,x_2]$. This gives the result
stated in the Proposition. \

We are left with the proof of the claim
in~\eqref{SC}.

For this, let us show that
\begin{equation}\label{SC2}
v >u \quad \mbox{on $(z_0,z_1]$} \quad
\text{and $z_0$ is strictly below $\J_0$.}\end{equation}

Call $\nu= (x_1-y_1)/|x_1-y_1|$ and observe that
\begin{equation}\label{11.11}
z_0=y_1+\Big(R-\frac{C}{R}-\sigma\Big)\nu,
\quad
z_1=y_1+\Big(R-\frac{C}{R}+1-\sigma\Big)\nu.\end{equation}
Hence,
making use of~\eqref{2.1a} and taking~$C$ large enough,
we have that
\begin{equation}\label{11}
[z_0,z_1] \subset B_{R_1}(y_1)\end{equation}
and
$$ z_0 \in
B_{R_0}.$$
Also, we know from Lemma \ref{easyHarnack}$(i)$ that
\begin{equation}\label{12}
\bar \beta_{R,\sigma}(x-y_1) \geq u (x)
\quad \text{in $B_{R_1}(y_1)$}.\end{equation}
Thus, \eqref{11} and \eqref{12} say that~\eqref{SC2}
will follow if we show that
\begin{equation}\label{SC3}
v(x) > \bar \beta_{R,\sigma}(x-y_1)
\quad \text{on $(B_{R_1}(y_1)\setminus B_{R_0}(y_1)) \cap
(z_0,z_1].$}\end{equation}
For this, we observe that,
by~\eqref{11.11},
$$v(x)= \sigma
+ s+\frac{C}{R} \quad \text{if $x=y_1+(s+R)\nu$}.$$ So, by
applying \eqref{11.11} and~\eqref{gbound}, if~$C$ suitably large,
we see that on $(z_0,z_1]$
\begin{eqnarray*}
&& v(x)-\bar \beta_{R,\sigma}(x-y_1)
=\sigma+s+\frac CR-\bar g(|x-y_1|-R)\\
&&\qquad=\sigma+s+\frac CR-\bar g(s)>0.
\end{eqnarray*}
This proves~\eqref{SC3} and so~\eqref{SC2}.

By replacing $z_i$ with $\xi_i$ in the proof of~\eqref{SC2},
one completes the proof of
the claim in~\eqref{SC}.
\end{proof}

The next Proposition also follows with similar arguments.

\begin{prop}\label{lsigma}Let $y_1,y_2$ be such that $u(y_1)=u(y_2)=1$ with $y_1\cdot e_1 > y_2 \cdot e_1$, and let
$\nu$ be the direction perpendicular to $y_1-y_2$ with $\nu_2>0.$
If \begin{equation}\label{J}
l_\sigma(y_1,y_2):= \{x \ | \
(x-y_1)\cdot \nu =\sigma\} \cap
\{x \ | \ x\cdot e_1 \in (y_1\cdot e_1, y_2 \cdot e_1)\},
\end{equation}
then
$$l_\sigma \subset \{u >\sigma\}.$$
\end{prop}

\begin{proof}
Consider the linear
one-dimensional solution $v$ to $F(D^2v)=0$ and~$|\nabla v|=1$
in~$\{0<v<1\}$
which
is
equal to 0 on the line $l_0$ connecting $y_1$ and $y_2$ and
increases in the $e_2$ direction. Denote by $l_1$ the 1 level set
of $v$ and by $z_1,z_2$ the intersection points of the vertical
segments through $y_1,y_2$ with $l_1$. Also, let $D$ be the open
parallelogram with vertices at $y_i$, $z_i$, $i=1,2$.

It suffices to show that $$u>v \quad \mbox{in $D$}$$ since
$$l_{\sigma}(y_1,y_2) \subset D$$ and $$v = \sigma \quad \text{on
$l_\sigma(y_1,y_2).$}$$

First let us notice that since $u=1$ above $\J_1$ then
\begin{eqnarray*}\label{vertside}&u > v
\quad \mbox{on $[y_i,z_i),$}\\\nonumber & z_i \quad \mbox{is
strictly above $\J_1$.}\end{eqnarray*} Now let us apply the
sliding method (upside-down). Since $u=1$ above $\J_1$, there
exists $t_0>0$ such that $u$ is above $v_{t}$ for all $t \geq
t_0$. Let $\bar t$ be the largest $t \leq t_0$ such that $u$
touches $v_{ t}$ from above. Assume by contradiction that $\bar t
\geq 0$. Again,
\begin{eqnarray}\label{vertside2}&u > v_{\bar t} \quad \mbox{on
$[y_i+\bar te_2,z_i+\bar te_2),$}\\\nonumber & z_i+\bar te_2 \quad
\mbox{is strictly above $\J_1$.}\end{eqnarray}

In particular $u$ and $v_{\bar t}$ cannot coincide. Thus,
similarly to the proof of Proposition \ref{convexhull}, the
comparison principle, Hopf
Lemma and \eqref{vertside2} imply that
$$u> v_{\bar t} \quad \mbox{on $(\overline D + \bar t e_2) \setminus [z_1
+ \bar t e_2, z_2 + \bar t e_2]$},$$ $$[z_1 + \bar t e_2, z_2 +
\bar t e_2] \Subset \{u=1\}.$$ Thus for small $\eps>0$, $u$ is
above $v_{\bar t -\eps}$ which contradicts the definition of $\bar
t.$ Hence $\bar t<0$ and $u$ is above $v=v_0$ in $\overline D.$ In
particular, $u> v$ on $\overline D
 \setminus [z_1, z_2]$ and hence  $u>\sigma$ on the $\sigma$ level set
$l_\sigma(y_1,y_2)$ of $v$.
\end{proof}

\

Now we are ready to show that $\{u=\sigma\}$ is included in a
strip. We start by defining $E_{R,\sigma}$ to be the convex hull
generated by $A_{R,\sigma}$.

{F}rom Proposition \ref{convexhull}, we have
that
\begin{equation}\label{83}
E_{R,\sigma} \subset \{u < \sigma\}.
\end{equation}

As a consequence of Lemma \ref{easyHarnack}$(ii)$, we obtain
the following Lemma.

\begin{lem}\label{intersect1} There exists $C$ large universal
constant such that in any vertical strip of width $2C$, $\{a-C
\leq x_1 \leq a+ C\}$ there exists a point on $\p E_{C,\sigma}$
that is at distance less than $2C$ from $\J_1$.
 \end{lem}

\begin{proof} Let $B_C(y)$ be the ball with center on the line
$x_1=a$ which is tangent to $\{u=\sigma\}$ at $x_0$, with $y \in
A_{C,\sigma}$. By Lemma \ref{easyHarnack} there exists $$x_1 \in
\J_1 \cap B_{K}(x_0).$$ Then the desired point is the intersection
of $[x_0,y]$ with $\p E_{C,\sigma}$, provided that $C > K.$
\end{proof}

\begin{lem}\label{halfplane}
If~$C$ is large enough,
the set $E:=E_{C,\sigma}$ is a half-plane.\end{lem}
\begin{proof}
Since $E$ is a convex set and for any $x_0 \in \R^2$ there exists
$t_0$ such that $x_0 + te_2 \in E$ for all $t \leq t_0$, then $\p
E$ is the graph of a concave function $f$ i.e.
\begin{equation}\label{limit}\p E = \{(s, f(s)), s \in \R\}.\end{equation}

If $f$ is not linear, we can find a linear function $p(s)$ tangent
to $f$ at some $s_0$ such that $$\lim_{s\to\pm\infty} (p(s)-
f(s))=+\infty.$$

{F}rom Lemma \ref{intersect1} for each $n\in \mathbb{N}$ there
exist points $$y_n \in \{x \ | \ (n-3)C \leq x \cdot e_1 \leq
(n+3)C\} \cap \J_1$$  $$z_n \in \{x \ | \ -(n+3)C \leq x \cdot e_1
\leq -(n-3)C\} \cap \J_1$$ that are at distance at most $2C$ from
the graph of $f$.

{F}rom \eqref{limit} we see that as $n \to \infty$, the point
$(s_0,
f(s_0))$ is at an arbitrary large distance above the line passing
through $y_n, z_n.$

Thus, when this distance becomes greater than $2\sigma$,
in the light of~\eqref{J},
we can find a point $$(s_0, t) \in l_\sigma(y_n,z_n), \quad
\mbox{for some $t < f(s_0)$}.$$
{F}rom Proposition \ref{lsigma} we
know that $$(s_0, t) \in \{u > \sigma\}.$$
On the other hand, since $t<f(s_0)$,
$$(s_0,t) \in E
\subset \{u < \sigma\}.$$

Thus we reached a contradiction.
This shows that~$f$ is a linear function.
\end{proof}

\begin{cor}\label{star}
$\{u=\sigma\}$ is contained in a strip. The direction $\nu$ of the
strip satisfies~$\nu_2\ne 0$. Furthermore, $\Omega$ is also
contained in a strip in the $\nu$ direction.
\end{cor}

\begin{proof} {F}rom \eqref{83},
Lemma \ref{halfplane} (and their counterpart in $\{u>\sigma\}$),
we conclude that $\{u=\sigma\}$ is contained in a strip, say $\{a
\le x\cdot \nu \le b\}$. Also, $\nu_2 \neq 0$ since $\{u=\sigma\}$
is trapped between $\J_0$ and $\J_1$ which are graphs in the $e_2$
direction.

Now, applying Lemma \ref{easyHarnack}$(i)$ together with its
counterpart in $\{u>\sigma\}$, with $R$ a fixed large constant, we
conclude that $\J_0$ is above the line $\{x \cdot \nu = a -C \}$
while $\J_1$ is below the line $\{x \cdot \nu = b+ C \}$, for some
$C$ universal. This concludes the proof of the corollary.
\end{proof}

Next, we show that the width of the strip obtained in
Corollary~\ref{star} is arbitrarily small and hence $\{u=\sigma\}$
is a line, which concludes the proof of Theorem \ref{main}.

As a preliminary step in this direction, we prove that $u$ is
monotone in the vertical direction, so to be able to apply our
Harnack inequality (this is also a nice consequence of
Corollary~\ref{star}).

\begin{prop} $u$ is monotone increasing in the $e_2$ direction.
\end{prop}

\begin{proof} Assume for simplicity that $u \in C^3(\Omega)$ and $F \in C^1.$

{F}rom Corollary~\ref{star} we have that~$\Omega$ is included in a
strip, say $$S:=\{-a \leq x \cdot \nu \leq a\}.$$

Notice that $$u_2:=\frac{\partial u}{\partial x_2} \geq 0 \quad
\text{near $\partial \Omega$}$$ and by Lipschitz continuity (see
Lemma \ref{Lip})$$u_2 \geq -K \quad \text{in $\Omega$}$$ with
$K>0$ universal. Also, $u_2$ solves the linearized equation
$$a_{ij}(x)(u_2)_{ij}=0 \quad \text{in $\Omega$}$$ with
$$a_{ij}(x)= F_{ij}(D^2u).$$

We compare $u_2$ with $W(x)= w(x\cdot \nu^\perp, x\cdot\nu)$ in
the domain $\Omega \cap \{|x\cdot \nu^\perp | < L\}$, where
$$ w(s,t)= K a^{-2}e^{-L\delta}(e^{-\delta s}+ e^{\delta s})(t^2-2a^2).$$
Here~$L>0$ is fixed and it will be taken large in the sequel.

Notice that
\begin{equation}\label{AB}{\mbox{
$W\le0$ in~$S$}}\end{equation} and also
$$\lambda |(D^2w)^+| - \Lambda|(D^2w)^-| > 0$$ if
$\delta$ is small depending on $\lambda, \Lambda$.

This implies
that
$$a_{ij}(x)W_{ij}(x)>0.$$
Since $${\mbox{$W \leq 0$ on $\p
\Omega,$}}$$ due to~\eqref{AB},
and $$W \leq - K \quad \text{in $S \cap
\{|x \cdot
\nu^\perp| =L\}$}$$ we conclude from the maximum principle that
$$u_2 \geq W \quad \text{in $\Omega \cap \{|x\cdot \nu^\perp| < L\}$}.$$
On $\{|x\cdot \nu^\perp| \le L/2\} \cap \Omega$ we obtain that
$u_2 \geq -4Ke^{-3\delta L/2}$ and by letting $L \to \infty$ we
obtain that in fact $u_2 \geq 0$ in~$\Omega$.

The smoothness assumptions on $u$ and $F$ can be removed with the
techniques of \cite{CC}.
\end{proof}

\ Now, we are ready to complete the proof of our main Theorem.

\begin{proof}[Proof of Theorem \ref{main}]
The idea of the proof is that once~$\{u=\sigma\}$
is included in some strip of width, say,~$d>0$,
then it is included in a smaller strip, say of width~$3d/4$.
{F}rom this, of course, it would follow that~$\{u=\sigma\}$
is a line.

Here are the details of
such an improvement of
flatness.

By Corollary~\ref{star}, we may assume that
\begin{equation}\label{Z}
\{u=\sigma\}
\subset \{0 < x \cdot \nu < d\},\end{equation}
for some direction
$\nu$ with
$\nu_2>0$.

Let $B_R(x_*)$ be a large ball with center on the line $t\nu$
which touches $\{u=\sigma\}$ from below at $x_0$.
Let $\nu^*= (x_0-x_*)/R$.

We distinguish
two cases.

\textit{Case 1.} $x_0\cdot \nu > d/2.$ In this case the point
$$y_0:= x_0 - \frac{2C}{R}\nu^*
$$ with $C$ as in Proposition \ref{convexhull}
satisfies
\begin{equation}\label{ZZ}
y_0\cdot\nu>\frac d 4\end{equation}
as long as $R$ is large enough.

By applying Proposition \ref{convexhull} to the ball $B_R(x_*)$
and all
the balls of radius $R$ tangent from below to $\{x \cdot \nu
=0\}$, which lie in~$\{ u<\sigma\}$ because of~\eqref{Z},
we find
that $\{u=\sigma \}$ is above the line $\{x \cdot \nu = y_0\}$.

So, by~\eqref{ZZ},\begin{equation}\label{AA1}{\mbox{
$\{u=\sigma \}$ is above the line $\{x
\cdot \nu = d/4\}.$}}\end{equation}

\

\textit{Case 2.} $x_0 \cdot \nu \leq d/2.$ In this case,
the
point
$$y_d:= x_0 + \frac{\bar C}{R}\nu^* + \frac{2C}{M- \bar
C/R}\nu^*$$
satisfies
\begin{equation}\label{CC}
y_d\cdot\nu<\frac{3d}4\end{equation}
as long as
$R$ and $M$ are large (again $C$ is as in Proposition
\ref{convexhull}).
Notice also that $y_d$ is in the ball of radius
$$ M-\frac{\bar C}{R}- \frac{C}{M-\bar C/R}$$ centered at
$(M\nu^*+x_0)$.

Also, from Lemma \ref{Harnack},
$$B_{M-\frac{\bar
C}{R}}(M\nu^*+x_0) \subset \{u>\sigma\}$$ if $R$ is large.

Therefore,
by applying Proposition \ref{convexhull} (upside-down) to the ball
$B_{M-\frac{\bar C}{R}}(M\nu^*+x_0)$  and all the balls of the
same radius $R$ tangent from above to $\{x \cdot \nu =d\}$ we find
that $\{u=\sigma \}$ is below the line $\{x \cdot \nu = y_d \cdot
\nu \}.$

This and~\eqref{CC} give that
\begin{equation}\label{AA2}{\mbox{
$\{u=\sigma \}$ is below the line $\{x
\cdot \nu = 3d/4\}.$}}\end{equation}

\

In either cases, from~\eqref{Z}
and either~\eqref{AA1} or~\eqref{AA2},
we obtain that
$\{u=\sigma\}$ is
included in a
strip of width $3d/4$, which is the desired improvement of
flatness.
\end{proof}


\begin{thebibliography}{9999}

\bibitem[AC]{AC} H. W. Alt and L. A. Caffarelli, {\it Existence and regularity
for a minimum problem with free boundary}, J. Reine Angew. Math.,
325:105--144, 1981.

\bibitem[CC]{CC} L. A. Caffarelli and X. Cabre, {\it Fully nonlinear
elliptic equations}, Volume 43 of American Mathematical Society
Colloquium Publications, American Mathematical Society,
Providence, RI, 1995.

\bibitem[CS]{CS}
L.A. Caffarelli and S. Salsa, \emph{A geometric approach to free
boundary problems}, GSM 68, American Mathematical Society,
Providence, Rhode Island, 2005.

\bibitem[DS]{DS} D. De Silva and O. Savin, {\it Symmetry of global
solutions to a class of fully nonlinear elliptic equations in 2D},
To appear in Indiana Univ. Math. J., 2008.

\bibitem[DM]{DM} J. Dolbeault and R. Monneau, {\it
On a Liouville type theorem for isotropic homogeneous fully
nonlinear elliptic equations in dimension two},
Ann. Sc. Norm. Super. Pisa Cl. Sci. (5) 2(1):181--197, 2003.

\bibitem[FV]{FV} A. Farina and E. Valdinoci,  {\it Flattening results
for elliptic PDEs in unbounded domains with applications to
overdetermined problems,} Preprint, 2008.

\bibitem[S]{S} O. Savin, {\it Phase transitions: Regularity of at level
sets,} Ph.D thesis, Univeristy of Texas at Austin 2003. To appear
in Ann. of Math., January 2009.


\bibitem[V1]{V1} E. Valdinoci, {\it Bernoulli jets and the zero mean
curvature equation}, J. Differential Equations, 225(2):710--736,
2006.

\bibitem[V2]{V2} E. Valdinoci, {\it Flatness of Bernoulli jets}, Math.
Z., 254(2):257--298, 2006.

\end{thebibliography}
\end{document}